\documentclass[12pt,leqno]{amsart}
\usepackage{amssymb}
\usepackage{enumerate}
\usepackage{mathrsfs}
\usepackage{graphicx,color}
\theoremstyle{plain} 
\newtheorem{theorem}{Theorem}[section] 
 
\newtheorem{proposition}[theorem]{Proposition}
\newtheorem{lemma}[theorem]{Lemma} 

\theoremstyle{definition} 

\newtheorem{remark}[theorem]{Remark}

\DeclareMathOperator{\supp}{spt}

\DeclareMathOperator{\dist}{dist}
\DeclareMathOperator{\tr}{tr}
\DeclareMathOperator{\Div}{div}

\newcommand{\R}{\mathbb{R}}

\usepackage{constants}
\newconstantfamily{const}{symbol=C}
\numberwithin{equation}{section}
\usepackage{bm}
\renewcommand{\vec}[1]{\bm{#1}}

\title{Gradient estimates for mean curvature flow with Neumann boundary
conditions} 
\author{Masashi Mizuno}
\address[Masashi Mizuno]{Department of Mathematics, College of Science
and Technology, Nihon University, Tokyo 101-8308 JAPAN}
\email{mizuno@math.cst.nihon-u.ac.jp}
\author{Keisuke Takasao}
\address[Keisuke Takasao]{Graduate School of Mathematical Sciences, University of Tokyo, Komaba 3-8-1, Meguro-ku, Tokyo 153-8914, JAPAN}
\email{takasao@ms.u-tokyo.ac.jp}

\thanks{This work was supported by JSPS KAKENHI Grant Numbers 25800084,
25247008.}


\allowdisplaybreaks[1]
\begin{document}

\begin{abstract}
We study the mean curvature flow of graphs both with Neumann boundary
conditions and transport terms. We derive boundary gradient estimates
for the mean curvature flow. As an application, the existence of the
mean curvature flow of graphs is presented. A key argument is a boundary
monotonicity formula of a Huisken type derived using reflected backward
heat kernels. Furthermore, we provide regularity conditions for the
transport terms.
\end{abstract}

\maketitle

\section{Introduction}

We consider the mean curvature flow of graphs with transport terms and Neumann
boundary conditions:
\begin{equation}
 \label{eq:1.1}
  \left\{
 \begin{aligned}
  \frac{\partial_tu}{\sqrt{1+|du|^2}}
  &=\Div\left(\frac{du}{\sqrt{1+|du|^2}}\right)
  +\vec{f}(x,u,t)\cdot
  \vec{n}&\quad &x\in\Omega,\ t>0, \\
  du\cdot\nu\big|_{\partial\Omega}&=0,& &t>0, \\
  u(x,0)&=u_0(x),& &x\in\Omega,
  \end{aligned}
  \right.
\end{equation}
where $\Omega\subset\R^n$ is a bounded domain with a smooth boundary,
$\nu$ is an outer unit normal vector on $\partial\Omega$,
$u=u(x,t):\Omega\times[0,\infty)\rightarrow\R$ is an unknown function,
$du:=(\partial_{x_1}u,\ldots,\partial_{x_n}u)$ is the tangential
gradient of $u$, $u_0=u_0(x):\Omega\rightarrow\R$ is given initial
data, $\vec{f}:\Omega\times\R\times[0,\infty)\rightarrow\R^{n+1}$ is
a given transport term, and $\vec{n}=\frac1{\sqrt{1+|du|^2}}(-du,1)$. 
For a
solution $u$ of \eqref{eq:1.1} and $t>0$, the graph of $u(x,t)$, which
is
\begin{equation}
 \label{eq:1.2}
 \Gamma_t:=\{(x,u(x,t)):x\in\Omega\},
\end{equation}
satisfies the mean curvature flow with the transport term, which is
subjected to right angle boundary conditions given by
\begin{equation}
 \label{eq:1.3}
 \left\{
  \begin{aligned}
   \vec{V}&=\vec{H}+(\vec{f}\cdot\vec{n})\vec{n},&\quad \text{on}\
   \Gamma_t,\ &t>0, \\
   \Gamma_t&\perp\partial(\Omega\times \mathbb{R}),& &t>0, \\
  \end{aligned}
 \right.
\end{equation}
where
$\vec{n}:=\frac1{\sqrt{1+|du|^2}}(-du,1)$ is the unit normal vector of
$\Gamma_t$, 
$\vec{V}:=\frac{\partial_tu}{\sqrt{1+|du|^2}}\vec{n}$ is the normal
velocity vector of $\Gamma_t$, and
$\vec{H}:=\Div(\frac{du}{\sqrt{1+|du|^2}})\vec{n}$ is the mean curvature
vector of $\Gamma_t$ (see Figure~\ref{fig:1.1}). 
It is interesting to derive the regularity criterion of the transport
term to obtain the classical solution of
\eqref{eq:1.1}. 
Liu-Sato-Tonegawa~\cite{MR2956324} studied the following incompressible and viscous non-Newtonian two-phase fluid flow:
\begin{equation}
\label{eq:1.4}
 \left\{
  \begin{aligned}
   &\frac{\partial \vec{f}}{\partial t} +\vec{f} \cdot \nabla \vec{f} 
   = \Div ( T^{\pm} (\vec{f},\Pi)) , \ \Div \vec{f} =0, 
   & \quad \text{on} \ \Omega ^{\pm} _t , \ & t>0,\\
   &\vec{n} \cdot (T^+ (\vec{f} ,\Pi) -T^- (\vec{f} ,\Pi)) =\vec{H},
   &\quad \text{on} \ \Gamma_t,\ &t>0, \\
   &\vec{V}=\vec{H}+(\vec{f}\cdot\vec{n})\vec{n},&\quad \text{on}\
   \Gamma_t,\ &t>0,
  \end{aligned}
 \right.
\end{equation}
where $\Omega ^+ _t \cup \Omega ^-  _t \cup \Gamma _t =\mathbb{T}^{n+1}=(\mathbb{R} / \mathbb{Z})^{n+1}$, $\vec{f}$ is the velocity vector of the fluids, $T^\pm$ is the stress tensor of the fluids, and $\Pi$ is the pressure of the fluids. The physical background of \eqref{eq:1.4} was studied by Liu-Walkington~\cite{MR1857673}.
The phase
boundary $\Gamma _t$ moves by the fluid flow and its mean curvature. In
\eqref{eq:1.3}, the transport term is corresponding to the fluid
velocity of \eqref{eq:1.4}. Since the regularity of non-Newtonian fluid flow is still
difficult problems, it is important to study regularity conditions to
obtain the classical solution of \eqref{eq:1.1}.

\begin{figure}
 \centering
 \includegraphics[height=120pt]{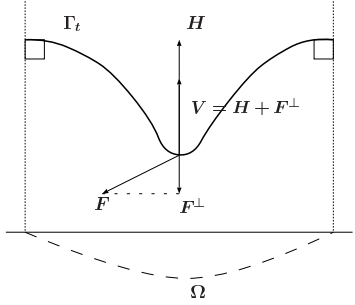}
 \caption{Mean curvature flow with the transport $\vec{F}$ and with the
 Neumann boundary condition.}
 \label{fig:1.1}
\end{figure}

To study the behavior of $\Gamma_t$, we need to investigate
$v:=\sqrt{1+|du|^2}$, which is the volume element of $\Gamma_t$. Thus,
it is important to derive gradient estimates for
\eqref{eq:1.1}. Interior gradient estimates for \eqref{eq:1.1} under
$\vec{f}\equiv0$ were studied by Ecker-Huisken~\cite{MR1117150} when the
initial surface is $C^1$, and by Colding-Minicozzi II~\cite{MR2099114}
when $u_0$ is bounded. It is difficult to apply their arguments to
\eqref{eq:1.1} under non-smooth transport terms because their arguments
essentially use the comparison arguments.

Ecker-Huisken~\cite{MR1025164} also derived the interior gradient
estimates for \eqref{eq:1.1} under $\vec{f}\equiv0$. In their arguments,
monotonicity formula is crucial to show the gradient estimates.
Takasao~\cite{MR3058702} studied the interior gradient estimates for
\eqref{eq:1.1} when $u_0$ is $C^1$ and the transport $\vec{f}$ is
bounded in time and space variables. An essential part of Takasao's
proof is to derive the monotonicity formula of the Huisken type with the
bounded transport terms.

Huisken~\cite{MR0983300} studied \eqref{eq:1.1} with the Neumann
boundary condition and without the transport $\vec{f}$. He showed the
existence of a classical solution of \eqref{eq:1.1} under
$\vec{f}\equiv0$. To show the existence of the solution, it is important
to derive up-to-boundary a priori gradient estimates of
\eqref{eq:1.1}. Huisken showed the gradient estimates when the initial
data $u_0$ is $C^{2,\alpha}$ up to boundary and $\partial\Omega$ is of
class $C^{2,\alpha}$ via the Schauder estimates. Stahl~\cite{MR1393271}
also considered the gradient estimates of \eqref{eq:1.1} without the
transport and obtained some blow-up criterion of the classical solution
of~\eqref{eq:1.1} under $\vec{f}\equiv0$. Up-to-boundary a priori
gradient estimates of the mean curvature flow with the Neumann boundary
condition are studied by many researchers and we mention
\cite{MR1384396, MR2180601, MR3479560, 1602.03614, MR2886118,
MR1402731, MR1393271, MR1425580, MR2691040, 1405.7774, MR3200337,
MR3150205}.
Our arguments are similar to Ecker's or Takasao's work~\cite{MR2024995,
MR3058702}. In our setting, we need to derive a boundary monotonicity
formula for \eqref{eq:1.1}. From this point, Buckland~\cite{MR2180601}
obtained the boundary monotonicity formula for \eqref{eq:1.1} under
$\vec{f}\equiv0$. Takasao derived the monotonicity formula for
\eqref{eq:1.1} with transport terms in ~\cite{MR3058702} but the
condition was not optimal.  On the other hand, reasonable conditions for
the transport terms for the regularity of weak mean curvature flow were
obtained in \cite{MR3194675, MR3176585}. In this paper, we obtain a
priori gradient estimates with reasonable conditions for the transport
terms.

Our problem \eqref{eq:1.1} imposes Neumann boundary conditions; thus,
up-to-the-boundary gradient estimates are also important. Mizuno and
Tonegawa~\cite{MR3348119} studied weak mean curvature flow with Neumann
boundary conditions via phase field methods. To study boundary behavior,
it was important to derive an $\varepsilon$-diffused boundary
monotonicity formula of a Huisken type via reflected backward heat
kernels (cf. \cite{MR1030675}, \cite{MR1237490}). Thus, it is also
important to derive the boundary monotonicity formula for \eqref{eq:1.1}
and determine the optimal regularity condition for the transport
terms. In this paper, we derive the boundary monotonicity formula for
\eqref{eq:1.1} and as an application, we derive a priori boundary
gradient estimates and prove the existence of a classical solution of
\eqref{eq:1.1}.

This paper is organized as follows. In section \ref{sec:2}, we present
basic notation and the main results. In section \ref{sec:3}, we derive
the boundary monotonicity formula for \eqref{eq:1.1}. In section
\ref{sec:4}, we derive the up-to-boundary gradient estimates for
\eqref{eq:1.1} and some integral estimates for the transport terms. In
section \ref{sec:5}, we prove the existence of the classical solution of
\eqref{eq:1.1}. In section \ref{sec:6}, we discuss some optimality for
the transport terms to obtain the gradient estimates for \eqref{eq:1.1}.

\section{Preliminaries and main results}
\label{sec:2}

\subsection{Notation}
Let $\vec{\nu}$ be an outer unit normal vector on
$\partial(\Omega\times\R)$; $\vec{\nu}=(\nu,0)$. For $n$-dimensional
symmetric matrices $A$ and $B$, define the inner product $A:B$ as
$A:B=\tr(AB)$.  Set $Q_T:=\Omega \times (0,T)$ and $Q_T^\varepsilon
:=\Omega \times (\varepsilon,T)$ for $0<\varepsilon <T$. Let $d$ and $D$
be the gradients on $\Omega$ and $\Omega\times\R$, respectively. Let
$D_{\Gamma_t}$ and $\Delta_{\Gamma_t}$ be the covariant differentiation and
Laplace-Beltrami operator on $\Gamma_t$, respectively. For a solution
$u$ of \eqref{eq:1.1}, let
$h:=-\Div\Big(\frac{du}{\sqrt{1+|du|^2}}\Big)$,
$v:=\sqrt{1+|du|^2}$. Then, equation \eqref{eq:1.1} becomes
\begin{equation}
 \label{eq:2.4}
  \partial_tu=-vh+(\vec{f}\cdot\vec{n})v.
\end{equation}

\subsection{Main results}

Let $T_0>0$ be fixed. We impose a regularity assumption on the transport term
such that for $p,q\geq1$ satisfying $\frac{n}p+\frac2q<1$,
\begin{equation}
 \label{eq:2.3}
  \|\vec{f}\|_{L^q_tL^p_x(0,T_0,\Gamma_t)}
  :=
  \left(
   \int_0^{T_0}\left(
		\int_{\Gamma_t}|\vec{f}(X,t)|^p\,d\mathscr{H}^n
	       \right)^\frac{q}{p}
   dt
  \right)^\frac1q
  <\infty.
\end{equation}

\begin{remark}
 Using the Meyer-Ziemer inequality (cf. 
 \cite[p. 266, Theorem 5.12.4]{MR1014685}),
 \[
 \int_{\Gamma_t}|\vec{f}(X,t)|^p\,d\mathscr{H}^n
 \leq C\|\vec{f}(\cdot,t)\|^p_{W^{1,p}(\Omega\times\R)};
 \]
 hence, our assumption \eqref{eq:2.3} is fulfilled if
 $\vec{f}\in L^q([0,T_0]:W^{1,p}(\Omega\times\R))$.
\end{remark}

First we derive a priori gradient estimates for \eqref{eq:1.1}.

\begin{theorem}[A priori estimates for the gradient]
 \label{thm:1} 
 Let $u$ be a classical solution of \eqref{eq:1.1} on
 $\Omega\times(0,T_0)$. Assume that $\Omega$ is convex, $u_0\in
 W^{1,\infty}(\Omega)$, and the transport term $\vec{f}$ satisfies
 \eqref{eq:2.3}. Then there exists $T>0$
 depending only on $n$, $p$, $q$, $T_0$, $\|du_0\|_\infty$,
 $\|\vec{f}\|_{L^q_tL^p_x(0,T_0,\Gamma_t)}$, and $\Omega$
 such that
 \begin{equation}
  \label{eq:2.2}
   \sup_{0<t<T,x\in\overline{\Omega}}\sqrt{1+|du(x,t)|^2}
   \leq 4
   (1+\|du_0\|_\infty^2).
 \end{equation}
\end{theorem}

The regularity assumption \eqref{eq:2.3} is reasonable from
blow-up arguments. Indeed, using the scale transform
\[
 x=\lambda y,\quad t=\lambda^2s,\quad w(y,s)=\frac1\lambda u(x,t),
\]
we obtain
\[
 \left\{
  \begin{aligned}
   \frac{\partial_sw}{\sqrt{1+|dw|^2}}
   &=\Div\left(\frac{dw}{\sqrt{1+|dw|^2}}\right)
  +\lambda\vec{f}(\lambda y,\lambda w,\lambda^2s)\cdot\dfrac{(-dw,1)}{\sqrt{1+|dw|^2}},
\\
   dw&=du.
  \end{aligned}
  \right.
\]
Then
\[
 \|\lambda\vec{f}(\lambda y,\lambda w,\lambda^2s)\|_{L^q_sL^p_y}
 =\lambda^{1-\frac{n}{p}-\frac2q}\|\vec{f}\|_{L^q_tL^p_x}
\]
and $\|\lambda\vec{f}(\lambda y,\lambda
w,\lambda^2s)\|_{L^q_sL^p_y}\rightarrow0$ as $\lambda\rightarrow0$ if
\eqref{eq:2.3} is fulfilled; that is, the transport is a small
perturbation for blow-up arguments. Note that the regularity assumption
\eqref{eq:2.3} is the same as the assumption for the parabolic Allard's
regularity theory developed by
Kasai-Tonegawa~\cite{MR3194675,MR3176585}. Furthermore, our results
include results from the study by Takasao~\cite{MR3058702} because our
argument also applies to interior gradient estimates. We further explain
in Section \ref{sec:6}
if \eqref{eq:2.3} is fulfilled for $\frac{n}{p}+\frac{2}{q}>1$, then
there is a solution of \eqref{eq:1.1} such that the gradient of the
solution is unbounded.

From the regularity estimate \eqref{eq:2.2}, the graph $\Gamma_t$
subjected to \eqref{eq:1.2} is a $C^1$-Riemannian manifold up to the
boundary. Furthermore, the graph $\Gamma_t$ is perpendicular to
$\partial\Omega\times\R$, which is the boundary of a cylinder
$\Omega\times\R$. In terms of partial differential equations, Theorem
\ref{thm:1} can be regarded as an up-to-the-boundary parabolic smoothing
effect for
$\partial_tu-\sqrt{1+|du|^2}\Div\Bigl(\frac{du}{\sqrt{1+|du|^2}}\Bigr)$.
The non-divergence elliptic differential operator
$-\sqrt{1+|du|^2}\Div\Bigl(\frac{du}{\sqrt{1+|du|^2}}\Bigr)$ is
degenerate hence regularity for solutions of \eqref{eq:1.1} is not
clear. When the gradient of solutions is bounded, then the Schauder
estimates for \eqref{eq:1.1} is applicable thus the higher regularity of
solutions and the existence of a solution of \eqref{eq:1.1} can be
deduced. Theorem \ref{thm:1} also can be regarded as a parabolic
smoothing effect for the mean curvature operator. To summarize,
\eqref{eq:2.2} determines how we obtain regularity of the mean curvature
flow.

To obtain the gradient estimates via the comparison arguments, the
boundedness of $\|D\vec{f}\|_ \infty$ is needed. On the other hand, to
obtain the gradient estimates via the monotonicity formula of the
Huisken type, the boundedness of $\|D\vec{f}\|_ \infty$ is not needed.
Note that the idea using the weighted monotonicity formula is
called the noncompact maximum principle~\cite[Proposition 4.27]{MR2024995}.  To
show the up-to-boundary monotonicity formula, we introduce reflected
backward heat kernels to compute the boundary integrals and derive
integral estimates for the transport terms under the assumption
\eqref{eq:2.3}.

Next, we demonstrate the existence of a classical solution of
\eqref{eq:1.1}. We assume parabolic H\"older continuity for $\vec{f}$;
that is, there is $\alpha\in(0,1]$ such that
\begin{equation}
 \label{eq:2.1}
 K:= \sup_{(X,t),(Y,s) \in (\Omega \times \mathbb{R}) \times (0,T_0) } \frac{|\vec{f}(X,t) -\vec{f}(Y,s)|}{|X-Y|^\alpha +|t-s|^{\alpha/2}}< \infty .
\end{equation} 

\begin{theorem}[Existence of a classical solution]
 \label{thm:2}
 Assume that $\Omega$ is convex, $u_0\in W^{1,\infty}(\Omega)$ with
 $du_0\cdot\nu\equiv0$ on $\partial\Omega$ and the transport term
 $\vec{f}$ satisfies \eqref{eq:2.1} with some $\alpha\in(0,1]$.  Then,
 there exists a time local unique solution $u\in C(\overline{Q_T})\cap
 C^{2,\alpha}(Q_T^{\varepsilon})$ for all $\varepsilon>0$ of
 $(\ref{eq:1.1})$ with $u(0)=u_0$ for some $0<T<T_0$. Furthermore, for
 any $\varepsilon >0$ there exists $\Cl[const]{c-1.2}>0$  depending
 only on $n,\alpha,L ,K $ such that
 \begin{equation}
  \|u\|_{C^{ 2,\alpha}(Q _T^\varepsilon )} \leq  \Cr{c-1.2}.
   \label{schauder}
 \end{equation}
\end{theorem}

Theorem \ref{thm:2} is deduced from the Leray-Schauder fixed point
theorem for the linearized problem of \eqref{eq:1.1}. Theorem
\ref{thm:1} is employed as a priori gradient estimates for the
Leray-Schauder fixed point theorem. As a result of the gradient bounds,
the linearized problem of \eqref{eq:1.1}
can be computed in the same class as the uniformly elliptic operator;
hence, we can derive the Schauder estimates for \eqref{eq:1.1} and apply
the Leray-Schauder fixed point theorem.

\section{Monotonicity of the metric}
\label{sec:3}
Our first task is to establish the up-to-the-boundary
monotonicity formula of the Huisken type. 

\begin{lemma}
 \label{lem:3.2}
 Let $u$ be a classical solution of \eqref{eq:1.1} and
 $v:=\sqrt{1+|du|^2}$. Then
 \begin{equation}
  \label{eq:3.2}
  \partial_tv-\Delta_{\Gamma_t}v
   -\left(\frac{du}{v}\cdot dv\right)\frac{\partial_tu}{v}
   =-|A_t|^2v-\frac{2|D_{\Gamma_t}v|^2}{v}
   +du\cdot d(\vec{f}\cdot\vec{n}),
 \end{equation}
 where $D_{\Gamma_t}$, $\Delta_{\Gamma_t}$ and $|A_t|$ denote covariant
 differentiation in $\Gamma_t$, Laplace-Beltrami operator, norm of
 second fundamental form of $\Gamma_t$ respectively.
\end{lemma}

\begin{proof}
 According to Ecker-Huisken~\cite{MR0998603},
 \[
  -\Delta_{\Gamma_t}v+|A_t|^2v
 +\frac{2|D_{\Gamma_t}v|^2}{v}
 -v^2(D_{\Gamma_t}h\cdot \vec{e}_{n+1})=0,
 \]
 where $\vec{e}_{n+1}=(0,\ldots,0,1)$. Because
 \[
 \begin{split}
  v^2(D_{\Gamma_t}h\cdot \vec{e}_{n+1})
  &=v^2(Dh\cdot
  \vec{e}_{n+1}-(Dh\cdot\vec{n})(\vec{n}\cdot\vec{e}_{n+1})) \\
  &=dh\cdot du \\
  &=-d\left(\frac{\partial_tu}{v}\right)\cdot du
  +d(\vec{f}\cdot\vec{n})\cdot du \quad (\because \eqref{eq:1.1})\\
  &=-\partial_tv
  +\left(\frac{du}{v}\cdot dv\right)\frac{\partial_tu}{v}
  +d(\vec{f}\cdot\vec{n})\cdot du, 
 \end{split} 
\]
 we obtain \eqref{eq:3.2}.
\end{proof}

Let
\[
 R:=\frac{1}{\|\text{principal curvature of }\partial\Omega\|_{L^\infty(\partial\Omega)}}.
\]
Because $\partial\Omega$ is smooth and compact, $0<R<\infty$. For
$r<R$, let $N_r$ denote the interior tubular neighborhood of
$\partial\Omega$;
\[
 N_r:=\{x\in\Omega:\dist(x,\partial\Omega)<r\}.
\]
For $x\in N_r$, there uniquely exists $\zeta(x)\in\partial\Omega$ such
that $\dist(x,\partial\Omega)=|x-\zeta(x)|$. Thus, define the reflection
point $\tilde{x}$ with respect to $\partial\Omega$ as
$\tilde{x}=2\zeta(x)-x$ (see Figure \ref{fig:3.1}). We fix a radially
symmetric cut-off function $\eta_R=\eta_R (|X|)\in C^\infty(\R^{n+1})$
such that
\[
 0\leq \eta_R \leq 1,\quad
  \eta'_R\leq 0,\quad
 \supp\eta_R \subset B_{R/8},\quad
 \eta_R =1\ \text{on}\ B_{R/16}.
\]

\begin{figure}
 \centering
 \includegraphics[width=120pt]{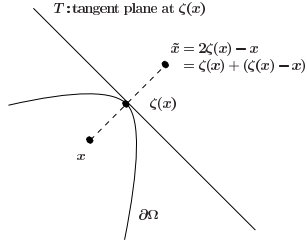}
 \caption{The reflection point of $x\in\Omega\cap N_r$ with respect to
 $\partial\Omega$ is denoted by $\tilde{x}$.}
 \label{fig:3.1}
\end{figure}

For $0<t<s<T_0$ and $X=(x,x_{n+1}),\ Y=(y,y_{n+1})\in N_R\times\R$, we
define the $n$-dimensional backward heat kernel $\rho_{(Y,s)}(X,t)$ and
reflected backward heat kernel $\tilde\rho_{(Y,s)}(X,t)$ as
\begin{equation}
 \begin{split}
  \rho_{(Y,s)}(X,t)
  :=\frac{1}{(4\pi(s-t))^\frac{n}{2}}
  \exp\left(-\frac{|X-Y|^2}{4(s-t)}\right), \\
  \tilde\rho_{(Y,s)}(X,t)
  :=\frac{1}{(4\pi(s-t))^\frac{n}{2}}
  \exp\left(-\frac{|\tilde{X}-Y|^2}{4(s-t)}\right), \\
 \end{split}
\end{equation}
where $\tilde{X}=(\tilde{x},x_{n+1})$. For fixed $0<t<s$ and $X,Y\in
N_R\times\R$, we define a truncated version of $\rho$ and $\tilde\rho$ as
\begin{equation}
 \begin{split}
  \rho_1
  &=\rho_1(X,t)
  :=\eta_R (X-Y)\rho_{(Y,s)}(X,t), \\
  \rho_2
  &=\rho_2(X,t)
  :=\eta_R (\tilde{X}-Y)\tilde\rho_{(Y,s)}(X,t).
 \end{split}
\end{equation}
 
 To derive Huisken's monotonicity formula,
 \begin{equation}
  \label{eq:3.4}
  \frac{(\vec{w}\cdot D\rho)^2}{\rho}
  +((I-\vec{w}\otimes\vec{w}):D^2\rho)
   +\partial_t\rho=0
 \end{equation}
 is the crucial identity, where $\rho=\rho_{(Y,s)}(X,t)$,
 $\vec{w}\in\R^{n+1}$ is any unit vector, $I$ is the identity matrix,
 $\vec{w}\otimes\vec{w}$ is tensor product, and
 $(I-\vec{w}\otimes\vec{w}):D^2\rho$ is
 $\tr((I-\vec{w}\otimes\vec{w})D^2\rho)$. In \cite{MR3348119}, a
 similar identity for the reflected backward heat kernel
 $\tilde\rho_{(Y,s)}$ was obtained.

\begin{lemma}%
 [\cite{MR3348119}]
 \label{lem:3.1}
 There is a constant $\Cl[const]{c-3.8}>0$ depending on $\Omega$
 such that for $\vec{w}=(w_i)\in\R^{n+1}$ with $|\vec{w}|=1$ and
 $\tilde\rho=\tilde\rho_{(Y,s)}(X,t)$,
 \begin{equation}
  \label{eq:3.5}
  \frac{(\vec{w}\cdot D\tilde{\rho})^2}{\tilde{\rho}}
  +((I-\vec{w}\otimes \vec{w}): D^2\tilde{\rho})
  +\partial_t\tilde{\rho} 
  \leq\Cr{c-3.8}\left(
  \frac{|\tilde{X}-Y|}{s-t}
  +\frac{|\tilde{X}-Y|^3}{(s-t)^2}
  \right)
   \tilde{\rho}
 \end{equation}
 for $0<t<s$ and $X,Y\in N_{R/2}\times\R$.
\end{lemma}

To prove Lemma \ref{lem:3.1}, we use the following lemma.
\begin{lemma}
 [cf.\cite{MR0397520,MR0863638}]
 Let 
 \[
 Q(X):=D\zeta(X)-(I-\vec{\nu}\otimes\vec{\nu}),
 \]
 where $\vec\nu$ is the unit normal vector at
 $\zeta(X)\in\partial\Omega$. Then
 \begin{enumerate}
  \item $Q(X)$ is symmetric;
  \item $Q(X)\vec{\nu}=\vec{0}$ for all $X\in N_{R/2}\times\R$;
  \item $Q(X)\vec{e}_{n+1}=\vec{0}$ for all $X\in N_{R/2}\times\R$, where $\vec{e}_{n+1}=(0,\ldots,0,1)$;
  \item $|Q(X)|\leq 2|X-\zeta(X)|$ for all $X\in N_{R/2}\times\R$;
  \item If $\partial\Omega\in C^3$, then $|D Q|$ is bounded.
 \end{enumerate}
\end{lemma}

For $X,Y\in N_{R/2}\times\R$, by convexity
\[
 |X-\zeta(X)|
 =\frac12|X-\tilde{X}|
 \leq\frac12(|X-Y|+|Y-\tilde{X}|)
 \leq|\tilde{X}-Y|,
\]
thus $|Q(X)|\leq 2|\tilde{X}-Y|$.

\begin{proof}
 [Proof of Lemma \ref{lem:3.1}]
 Since $D \zeta(X)=I-\vec{\nu}\otimes\vec{\nu} +Q(X)$ and $\tilde
 X=2\zeta(X)-X$, we have
 \begin{equation}
 \label{eq:3.15}
  \begin{split}
   D|\tilde{X}-Y|^2&=2(I-2\vec{\nu}\otimes\vec{\nu}+2Q(X))(\tilde{X}-Y), \\
   |D|\tilde{X}-Y|^2|^2&=4|(I-2\vec{\nu}\otimes\vec{\nu}+2Q(X))(\tilde{X}-Y)|^2 \\
   &\leq4|\tilde{X}-Y|^2+\Cr{c-3.9}|\tilde{X}-Y|^3, \\
   D_{ij} 
   |\tilde{X}-Y|^2&=2\delta_{ij}-4\sum_{k=1}^n (\partial_{X_j}(\nu_i
   \nu_k)-\partial_{X_j}q_{ik})(\tilde X_k-Y_k) \\
   &\quad+8q_{ij}+8\sum_{k=1}^{n}(q_{ik}q_{jk}-\nu_i\nu_kq_{jk}-\nu_j\nu_kq_{ik}),
  \end{split}
 \end{equation}
 where $Q(X)=(q_{ij})$ and $\Cl[const]{c-3.9}>0$ is some constant
 depending on $\Omega$. By direct calculation, we have
 \begin{equation}
 \label{eq:3.16}
  \begin{split}
   \partial_t\tilde\rho&=\left(\frac{n}{2(s-t)}-\frac{|\tilde{X}-Y|^2}{4(s-t)^2}\right)\tilde\rho, \ \ 
   D\tilde\rho=-\frac{D|\tilde{X}-Y|^2}{{4(s-t)}}\tilde\rho, \\
   D^2\tilde\rho
   &=\left(\frac{D|\tilde{X}-Y|^2\otimes D|\tilde{X}-Y|^2}{16(s-t)^2}
   -\frac{\tr(D^2|\tilde{X}-Y|^2)}{4(s-t)}
   \right)\tilde\rho.
  \end{split}
 \end{equation}
 Using \eqref{eq:3.15} and \eqref{eq:3.16}, we obtain \eqref{eq:3.5}.
\end{proof}

We next prove a weighted boundary monotonicity inequality.

\begin{lemma}
 \label{lem:3.4}
 Let $\phi\in C^1([0,\infty):C^2(\Omega))$ be a non-negative
 function. Then there exist positive numbers $\Cl[const]{c-3.1}$,
 $\Cl[const]{c-3.2}$ and $\Cl[const]{c-3.3}>0$ depending on $n$,
 $\Omega$ such that
 \begin{equation}
  \label{eq:3.10}
  \begin{split}
    \frac{d}{dt}\int_{\Gamma_t}\phi(\rho_1+\rho_2)\,d\mathscr{H}^n 
   &\leq \int_{\Gamma_t}(\rho_1+\rho_2)
   \left(\partial_t\phi-\Delta_{\Gamma_t}\phi-
   \left(d\phi\cdot\frac{du}{v}\right)
   \frac{\partial_tu}{v}\right)\,d\mathscr{H}^n \\
   &\quad
   +\frac14\int_{\Gamma_t}\phi(\rho_1+\rho_2)(\vec{f}\cdot\vec{n})^2
   \,d\mathscr{H}^n \\
   &\quad
   +\Cr{c-3.1}\mathscr{H}^n (\Gamma _t)+\Cr{c-3.2}(s-t)^{-\frac34}\int_{\Gamma_t}\phi\rho_2
   \,d\mathscr{H}^n \\
   &\quad
   +\Cr{c-3.3}\int_{\Gamma_t\cap\supp\rho_2}\phi
   \,d\mathscr{H}^n
   +\int_{\partial\Gamma_t}(\rho_1+\rho_2)(D_{\Gamma_t}\phi\cdot\vec{\nu})
   \,d\mathscr{H}^{n-1}.
  \end{split}
 \end{equation}
\end{lemma}

\begin{proof}
 For $i=1,2$
 \begin{equation}
  \label{eq:3.3}
  \begin{split}
   \frac{d}{dt}\int_{\Gamma_t}\phi\rho_i\,d\mathscr{H}^n
   &=\frac{d}{dt}\int_{\Omega}\phi(x,t)\rho_i(x,u(x,t),t)v(x,t)\,dx \\
   &=\int_{\Omega}
   \Bigl(
   \partial_t\phi(x,t)\rho_i(x,u(x,t),t) \\
   &\qquad
   +\phi(x,t)\partial_{x_{n+1}}\rho_i(x,u(x,t),t)\partial_tu(x,t)
   +\partial_{t}\rho_i(x,u(x,t),t)
   \Bigr)v(x,t)\,dx \\
   &\quad+\int_{\Omega}\phi(x,t)
   \rho_i(x,u(x,t),t)\partial_{t}v(x,t)\,dx \\
   &=\int_{\Gamma_t}(\partial_t\phi\rho_i
   +\phi\partial_t\rho_i
   +\phi\partial_{x_{n+1}}\rho_i\partial_tu)\,
   d\mathscr{H}^n \\
   &\quad+\int_{\Omega}\phi(x,t)\rho_i(x,u(x,t),t)\frac{du(x,t)\cdot d(\partial_tu(x,t))}{v(x,t)}\,dx. \\
  \end{split}
 \end{equation}

 We consider the last term of the equation
 \eqref{eq:3.3}. Using integration by parts, we obtain
 \[
 \begin{split}
  &\quad\int_{\Omega}\phi(x,t)\rho_i(x,u(x,t),t)
  \frac{du(x,t)\cdot d(\partial_tu(x,t))}{v(x,t)}\,dx \\
  &=-\int_\Omega\Div\left(
  \phi(x,t)\rho_i(x,u(x,t),t)\frac{du(x,t)}{v(x,t)}\right)
  \partial_tu(x,t)\,dx \\
  &=-\int_\Omega
  d\phi(x,t)\cdot\frac{du(x,t)}{v(x,t)}\rho_i(x,u(x,t),t)
  \frac{\partial_tu(x,t)}{v(x,t)}v(x,t)\,dx \\
  &\quad-\int_\Omega
  \phi(x,t)\bigl(d\rho_i(x,u(x,t),t) 
  +\partial_{x_{n+1}}\rho_i(x,u(x,t),t)du(x,t)\bigr)
  \cdot\frac{du(x,t)}{v(x,t)}
  \frac{\partial_tu(x,t)}{v(x,t)}v(x,t)\,dx\\
  &\quad-\int_\Omega
  \phi(x,t)\rho_i(x,u(x,t),t)\Div\left(\frac{du(x,t)}{v(x,t)}\right)
  \frac{\partial_tu(x,t)}{v(x,t)}v(x,t)\,dx \\
  &=-\int_{\Gamma_t}\Bigl(d\phi\cdot\frac{du}{v}\Bigr)
  \rho_i\frac{\partial_tu}{v}\,d\mathscr{H}^n
  -\int_{\Gamma_t}\phi
  \Bigl((d\rho_i+\partial_{x_{n+1}}\rho_idu)\cdot\frac{du}{v}\Bigr)
  \frac{\partial_tu}{v}\,d\mathscr{H}^n \\
  &\quad+\int_{\Gamma_t}\phi\rho_ih\frac{\partial_tu}{v}\,d\mathscr{H}^n.
 \end{split} 
 \]
We note that
 \[
  \begin{split}
   \partial_{{x_{n+1}}}\rho_iv
   -\partial_{{x_{n+1}}}\rho_i\frac{|du|^2}{v}
   -d\rho_i\cdot\frac{du}{v}
   &=\partial_{{x_{n+1}}}\rho_i\frac{1}{v}
   -d\rho_i\cdot\frac{du}{v}=(D\rho\cdot\vec{n}).
  \end{split}
 \]
 Hence, we obtain
 \[
 \begin{split}
  \frac{d}{dt}\int_{\Gamma_t}\phi\rho_i\,d\mathscr{H}^n
  &=\int_{\Gamma_t}\partial_t\phi\rho_i\,d\mathscr{H}^n
  +\int_{\Gamma_t}\phi\partial_t\rho_i\,d\mathscr{H}^n 
  +\int_{\Gamma_t}\phi(D\rho_i\cdot\vec{n})\frac{\partial_tu}{v}
  \,d\mathscr{H}^n \\
  &\quad-\int_{\Gamma_t}\Bigl(d\phi\cdot\frac{du}{v}\Bigr)\rho_i\frac{\partial_tu}{v}
  \,d\mathscr{H}^n
  +\int_{\Gamma_t}\phi\rho_ih\frac{\partial_tu}{v}\,d\mathscr{H}^n.
 \end{split}
 \]

 Using \eqref{eq:1.1} or \eqref{eq:2.4},
 \[
 \begin{split}
  \left(\phi(D\rho_i\cdot\vec{n})
  +\phi\rho_ih\right)\frac{\partial_tu}{v} 
  &=\left(\phi(D\rho_i\cdot\vec{n})
  +\phi\rho_ih\right)(-h+\vec{f}\cdot\vec{n}) \\
  &=\phi(D\rho_i\cdot\vec{H})
  -\phi\rho_i|\vec{H}|^2
  +\phi\rho_i\left(
  \left(\frac{D^\perp\rho_i}{\rho_i}-\vec{H}\right)\cdot\vec{n}\right)
  (\vec{f}\cdot\vec{n}) \\
  &=-\phi\rho_i\left|\vec{H}-\frac{D^\perp\rho_i}{\rho_i}\right|^2
  +\phi\frac{|D^\perp\rho_i|^2}{\rho_i}
  -\phi(D^\perp\rho_i\cdot\vec{H}) \\
  &\quad+\phi\rho_i\left(
  \left(\frac{D^\perp\rho_i}{\rho_i}-\vec{H}\right)\cdot\vec{n}\right)
  (\vec{f}\cdot\vec{n}) \\
  &\leq\phi\frac{|D^\perp\rho_i|^2}{\rho_i}
  -\phi(D^\perp\rho_i\cdot\vec{H})
  +\frac14\phi\rho_i(\vec{f}\cdot\vec{n})^2,
 \end{split}
 \]
where $\vec{H}=-h\vec{n}$ and $D^\perp \rho _i = (D\rho _i \cdot \vec{n}) \vec{n}$ are used. Therefore,
 \[
 \begin{split}
  \frac{d}{dt}\int_{\Gamma_t}\phi\rho_i\,d\mathscr{H}^n
  &\leq\int_{\Gamma_t}\partial_t\phi\rho_i\,d\mathscr{H}^n
  -\int_{\Gamma_t}\Bigl(d\phi\cdot\frac{du}{v}\Bigr)\rho_i\frac{\partial_tu}{v}
  \,d\mathscr{H}^n \\
  &\quad
  +\int_{\Gamma_t}\phi
  \left(\partial_t\rho_i+\frac{|D^\perp\rho_i|^2}{\rho_i}
  -(D^\perp\rho_i\cdot\vec{H})\right)
  \,d\mathscr{H}^n \\
  &\quad
  +\frac14\int_{\Gamma_t}\phi\rho_i(\vec{f}\cdot\vec{n})^2\,d\mathscr{H}^n.
  \end{split}
 \]
According to the divergence theorem on $\Gamma_t$, 
 \[
 \begin{split}
  -\int_{\Gamma_t}\phi(D^\perp\rho_i\cdot\vec{H})\,d\mathscr{H}^n 
  &=-\int_{\Gamma_t}\phi(D\rho_i\cdot\vec{H})\,d\mathscr{H}^n \\
  &=\int_{\Gamma_t}\Div_{\Gamma_t}(\phi D\rho_i)\,d\mathscr{H}^n
  -\int_{\partial\Gamma_t}\phi(D\rho_i\cdot\vec{\nu})\,d\mathscr{H}^{n-1} \\
  &=\int_{\Gamma_t}D_{\Gamma_t}\phi\cdot D\rho_i\,d\mathscr{H}^n
  +\int_{\Gamma_t}\phi
  ((I-\vec{n}\otimes\vec{n}):D^2\rho_i)\,d\mathscr{H}^n \\
  &\quad
  -\int_{\partial\Gamma_t}\phi(D\rho_i\cdot\vec{\nu})\,d\mathscr{H}^{n-1} \\
  &=-\int_{\Gamma_t}\rho_i\Delta_{\Gamma_t}\phi\,d\mathscr{H}^n
  +\int_{\Gamma_t}\phi
  ((I-\vec{n}\otimes\vec{n}):D^2\rho_i)\,d\mathscr{H}^n \\
  &\quad
  +\int_{\partial\Gamma_t}(\rho_i(D_{\Gamma_t}\phi\cdot\vec{\nu})-\phi(D\rho_i\cdot\vec{\nu}))\,d\mathscr{H}^{n-1}. \\
 \end{split} 
 \]
 Using \eqref{eq:3.4} and \eqref{eq:3.5}, we obtain
 \begin{equation}
  \label{eq:3.6}
  \frac{|D^\perp\rho_1|^2}{\rho_1}
   +((I-\vec{n}\otimes\vec{n}):D^2\rho_1)
   +\partial_t\rho_1\leq \Cr{c-3.4}
 \end{equation}
 and
 \begin{equation}
  \label{eq:3.7}
   \frac{|D^\perp\rho_2|^2}{\rho_2}
   +((I-\vec{n}\otimes\vec{n}):D^2\rho_2)
   +\partial_t\rho_2
   \leq
   \Cr{c-3.5}
    \left(
     \frac{|\tilde{X}-Y|}{s-t}
     +\frac{|\tilde{X}-Y|^3}{(s-t)^2}
    \right)
    \rho_2
    +\Cr{c-3.6}
 \end{equation}
 for some constants $\Cl[const]{c-3.4}$, $\Cl[const]{c-3.5}$,
 $\Cl[const]{c-3.6}>0$ depending on $\Omega$. 

 To compute the integration of \eqref{eq:3.7}, we decompose the
 integration as
 \[
 \begin{split}
  &\quad\int_{\Gamma_t}\phi\frac{\Cr{c-3.5}|\tilde{X}-Y|}{s-t}\rho_2\,d\mathscr{H}^n \\
  &\leq\int_{\Gamma_t\cap\{|\tilde{X}-Y|\leq(s-t)^\frac14\}}
  \phi\frac{\Cr{c-3.5}|\tilde{X}-Y|}{s-t}\rho_2\,d\mathscr{H}^n \\
  &\quad+\int_{\Gamma_t\cap\{|\tilde{X}-Y|\geq(s-t)^\frac14\}}
  \phi\frac{\Cr{c-3.5}|\tilde{X}-Y|}{s-t}\rho_2\,d\mathscr{H}^n \\
  &=:I_1+I_2,
 \end{split} 
 \]
 and
 \[
 \begin{split}
  &\quad\int_{\Gamma_t}\phi\frac{\Cr{c-3.5}|\tilde{X}-Y|^3}{(s-t)^2}\rho_2\,d\mathscr{H}^n \\
  &\leq\int_{\Gamma_t\cap\{|\tilde{X}-Y|\leq(s-t)^\frac{5}{12}\}}
  \phi\frac{\Cr{c-3.5}|\tilde{X}-Y|^3}{(s-t)^2}\rho_2\,d\mathscr{H}^n \\
  &\quad+\int_{\Gamma_t\cap\{|\tilde{X}-Y|\geq(s-t)^\frac{5}{12}\}}
  \phi\frac{\Cr{c-3.5}|\tilde{X}-Y|^3}{(s-t)^2}\rho_2\,d\mathscr{H}^n \\
  &=:I_3+I_4.
 \end{split} 
 \]
 $I_1$ is estimated by
 \begin{equation}
  \label{eq:3.8}
   I_1
   \leq \Cr{c-3.5}(s-t)^{-\frac34}
   \int_{\Gamma_t\cap\{|\tilde{X}-Y|\leq(s-t)^\frac14\}}
   \phi\rho_2\,d\mathscr{H}^n
   \leq \Cr{c-3.5}(s-t)^{-\frac34}\int_{\Gamma_t}\phi\rho_2\,d\mathscr{H}^n.
 \end{equation}
 $I_2$ is estimated by
 \begin{equation}
  \label{eq:3.9}
   I_2
   \leq\frac{\Cr{c-3.5}}{(s-t)^{1+\frac{n}{2}}}e^{-\frac{1}{4\sqrt{s-t}}}
   \int_{\Gamma_t\cap\supp\rho_2}
   \phi|\tilde{X}-Y|\,d\mathscr{H}^n
   \leq\Cr{c-3.7}R
   \int_{\Gamma_t\cap\supp\rho_2}
   \phi
   \,d\mathscr{H}^n
 \end{equation}
 for some constant $\Cl[const]{c-3.7}>0$ depending on $n$ and
 $\Omega$. $I_3$ and $I_4$ are estimated as a similar manner.
 
 Using \eqref{eq:3.8}, \eqref{eq:3.9},  
 $D(\rho_1+\rho_2)\cdot\vec{\nu}\big|_{\partial\Omega}\equiv0$, and $|\tilde{X}-Y|\leq R$ when $X \in \supp{\rho _2}$, we compute
 \[
 \begin{split}
  &\quad\frac{d}{dt}\int_{\Gamma_t}\phi(\rho_1+\rho_2)\,d\mathscr{H}^n \\
  &\leq\int_{\Gamma_t}(\rho_1+\rho_2)
  \left(\partial_t\phi-\Delta_{\Gamma_t}\phi
  -\left(d\phi\cdot\frac{du}{v}\right)\frac{\partial_tu}{v}\right)\,
  d\mathscr{H}^n \\
  &\quad+\frac14\int_{\Gamma_t}(\rho_1+\rho_2)
  \phi(\vec{f}\cdot\vec{n})^2\,d\mathscr{H}^n \\
  &\quad+(\Cr{c-3.4}+\Cr{c-3.6})\mathscr{H}^n(\Gamma_t)
  +2\Cr{c-3.5}(s-t)^{-\frac34}\int_{\Gamma_t}\phi\rho_2\,d\mathscr{H}^n \\
  &\quad+\Cr{c-3.7}(R+R^3)\int_{\Gamma_t\cap\supp\rho_2}
  \phi \,d\mathscr{H}^n 
  +\int_{\partial\Gamma_t}
  (\rho_1+\rho_2)(D_{\Gamma_t}\phi\cdot\vec{\nu})
  \,d\mathscr{H}^{n-1}. 
 \end{split}
\]
 For $\Cr{c-3.1}=\Cr{c-3.4}+\Cr{c-3.6}$, $\Cr{c-3.2}=2\Cr{c-3.5}$,
 and $\Cr{c-3.3}=\Cr{c-3.7}(R+R^3)$,
 we obtain \eqref{eq:3.10}.
\end{proof}

We use the following lemma to handle the boundary integral.

\begin{lemma}
 Let $u$ be a classical solution of \eqref{eq:1.1} and
  $v:=\sqrt{1+|du|^2}$. If $\Omega$ is convex, then
 \begin{equation}
  \label{eq:3.11}
  (D_{\Gamma_t}v\cdot\vec{\nu})|_{\partial(\Omega\times\R)}\leq0
 \end{equation}
 for all $t>0$.
\end{lemma}

\begin{proof}
 Because
 \[
 \begin{split}
  D_{\Gamma_t}v&=Dv-(Dv\cdot\vec{n})\vec{n} \\
  &=(dv,0)+\frac1{v^2}(dv\cdot du)(-du,1),
 \end{split}
 \]
 and boundary condition of $u$,
 \[
  \begin{split}
   (D_{\Gamma_t}v\cdot\vec{\nu})|_{\partial(\Omega\times\R)}
   &=\left((dv\cdot\nu)+\frac{1}{v^2}(dv\cdot du)(-du\cdot\nu)\right)\bigg|_{\partial\Omega} \\
   &=\frac{1}{2v}d|du|^2\cdot\nu|_{\partial\Omega} \\
   &=\frac{1}{v}B(du,du)|_{\partial\Omega}, 
  \end{split}
 \]
 where $B$ is the second fundamental form of $\partial\Omega$. Because
 of the convexity of $\Omega$, $B(du,du)\leq0$.
\end{proof}

Using \eqref{eq:3.2}, \eqref{eq:3.10}, and \eqref{eq:3.11}, monotonicity
of the metric is obtained as follows:
 \begin{proposition}
  \label{prop:3.6}
  Let $u$ be a classical solution of \eqref{eq:1.1} and
  $v:=\sqrt{1+|du|^2}$. Then
  for $Y\in N_{R/4}\times\R$ and $0<t<s$,
 \begin{equation}
  \label{eq:3.12}
  \begin{split}
   &\quad\frac{d}{dt}\int_{\Gamma_t}v(\rho_1+\rho_2)\,d\mathscr{H}^n \\
   &\leq -\int_{\Gamma_t}(\rho_1+\rho_2)
   \left(|A_t|^2v
   +\frac{2|D_{\Gamma_t}v|^2}{v}
   -du\cdot d(\vec{f}\cdot\vec{n})\right)
   \,d\mathscr{H}^n \\
   &\quad
   +\frac14\int_{\Gamma_t}v(\rho_1+\rho_2)(\vec{f}\cdot\vec{n})^2
   \,d\mathscr{H}^n \\
   &\quad
   +\Cr{c-3.1}\mathscr{H}^n(\Gamma _t) +\Cr{c-3.2}(s-t)^{-\frac34}\int_{\Gamma_t}v(\rho_1+\rho_2)
   \,d\mathscr{H}^n \\
   &\quad
   +\Cr{c-3.3}\int_{\Gamma_t\cap\supp\rho_2}v
   \,d\mathscr{H}^n 
  \end{split}
 \end{equation}
  where $\Cr{c-3.1}$, $\Cr{c-3.2}$, $\Cr{c-3.3}$ are constants as in
  Lemma \ref{lem:3.4}.
 \end{proposition}

\section{Gradient estimates}
\label{sec:4}

We deduce the integral estimates for the transport terms.

\begin{lemma}
 \label{lem:4.1}
 Let $\vec{f}$ be in $ L^q_tL^p_x(0,T_0,\Gamma_t)$ with
 $1-\frac{n}{p}-\frac2q>0$. Let $u$ be a classical solution of
 \eqref{eq:1.1} and $v:=\sqrt{1+|du|^2}$. Let $\eta\in L^\infty(0,T_0)$
 be a nonnegative function.  Then there is a constant
 $\Cl[const]{c-4.1}>0$
 depending
 only on $n$, $p$, $q$ and $T_0$ such that
 \begin{equation}
  \label{eq:4.1}
 \begin{split}
  &\qquad\int_0^\tau \eta\,dt\int_{\Gamma_t}(\rho_1+\rho_2)
  du\cdot d(\vec{f}\cdot\vec{n})
  \,d\mathscr{H}^n \\
  &\quad+\frac14\int_0^\tau \eta\, dt\int_{\Gamma_t}v(\rho_1+\rho_2)(\vec{f}\cdot\vec{n})^2
  \,d\mathscr{H}^n \\
  &\leq \frac12
  \int_0^\tau\eta\,dt
  \int_{\Gamma_t}
  (\rho_1+\rho_2)|A_t|^2v
  \,d\mathscr{H}^n 
  +\int_0^\tau\eta\,dt\int_{\Gamma_t}
  (\rho_1+\rho_2)\frac{|D_{\Gamma_t}v|^2}{v}
  \,d\mathscr{H}^n \\
  &\quad+\Cr{c-4.1}
  \|\eta\|_{L^\infty(0,T_0)}
  \|v\|_{L^\infty(\Omega\times(0,\tau))}^3
  \|\vec{f}\|_{L^q_tL^p_x(0,\tau,\Gamma_t)}
  (1+\|\vec{f}\|_{L^q_tL^p_x(0,\tau,\Gamma_t)})
  \end{split} 
 \end{equation}
 for $0<\tau<s$, where
 \[
 \|\vec{f}\|_{L^q_tL^p_x(0,\tau,\Gamma_t)}
 :=
 \left(
 \int_0^\tau\left(
 \int_{\Gamma_t}|\vec{f}(X,t)|^p\,d\mathscr{H}^n
 \right)^\frac{q}{p}
 dt
 \right)^\frac1q.
 \]
\end{lemma}

\begin{proof}
 For simplicity, set $\bar{\rho}:=\rho_1+\rho_2$. Then
 \[
 \begin{split}
  &\quad\int_{\Gamma_t}(\rho_1+\rho_2)
  (du\cdot d(\vec{f}\cdot\vec{n}))
  \,d\mathscr{H}^n \\
  &=\int_{\Omega}\bar{\rho}
  (du\cdot d(\vec{f}\cdot\vec{n}))
  v\,dx \\
  &=-\int_{\Omega}
  (\bar{\rho}\Delta uv+(du\cdot d(\bar{\rho}(x,u,t)))v+\bar{\rho}(du\cdot dv))
  (\vec{f}\cdot\vec{n})\,dx \\
  &=-\int_{\Gamma_t}
  \biggl(\bar{\rho}\Delta u+(du\cdot d(\bar{\rho}(x,u,t)))
  +\bar{\rho}\Bigl(\frac{du}{v}\cdot dv\Bigr)\biggr)
  (\vec{f}\cdot\vec{n})\,d\mathscr{H}^n.
 \end{split}
 \]
 Here
 \[
  h=-\Div\left(\frac{du}{v}\right)=-\frac1{v}\Delta
 u+\frac{1}{v^2}(du\cdot dv);
 \]
 hence,
 \[
 \begin{split}
  \int_{\Gamma_t}\bar{\rho}
  du\cdot d(\vec{f}\cdot\vec{n})
  \,d\mathscr{H}^n
  &=\int_{\Gamma_t}\bar{\rho} vh(\vec{f}\cdot\vec{n})\,d\mathscr{H}^n \\
  &\quad-2\int_{\Gamma_t}\bar{\rho} 
  \Bigl(\frac{du}{v}\cdot dv\Bigr)
  (\vec{f}\cdot\vec{n})\,d\mathscr{H}^n \\
  &\quad-\int_{\Gamma_t}(du\cdot d(\bar{\rho}(x,u,t)))
  (\vec{f}\cdot\vec{n})\,d\mathscr{H}^n \\
  &=:I_1+I_2+I_3.
 \end{split}
 \]

 $I_1$ is estimated by
 \begin{equation}
  \label{eq:4.2}
  \begin{split}
   |I_1|
   &\leq\frac1{2n}\int_{\Gamma_t}\bar{\rho} h^2v\,d\mathscr{H}^n
   +\frac{n}{2}\int_{\Gamma_t}\bar{\rho}
   v(\vec{f}\cdot\vec{n})^2\,d\mathscr{H}^n \\
   &\leq\frac1{2}\int_{\Gamma_t}\bar{\rho} |A_t|^2v\,d\mathscr{H}^n
   +\frac{n}{2}\int_{\Gamma_t}\bar{\rho}
   v(\vec{f}\cdot\vec{n})^2\,d\mathscr{H}^n 
  \end{split} 
 \end{equation}
because $h^2\leq n|A_t|^2$.

 Note that $D_{\Gamma_t}v=Dv-(Dv\cdot\vec{n})\vec{n}$,
 \[
 \begin{split}
  |D_{\Gamma_t}v|^2
  &=|Dv|^2-(Dv\cdot\vec{n})^2 \\
  &=|dv|^2-\frac{1}{v^2}(du\cdot dv)^2 \quad (\because Dv=(dv,0))\\
  &\geq|dv|^2-\frac{1}{v^2}|du|^2|dv|^2 \\
  &=\frac1{v^2}|dv|^2.
 \end{split} 
 \]
Therefore,
 \begin{equation}
  \label{eq:4.3}
  \begin{split}
   |I_2|
   &\leq\int_{\Gamma_t}\bar{\rho}
   \frac{|du|^2|dv|^2}{v^5}\,d\mathscr{H}^n
   +\int_{\Gamma_t}\bar{\rho} v^3(\vec{f}\cdot\vec{n})^2\,d\mathscr{H}^n \\
   &\leq\int_{\Gamma_t}\bar{\rho}
   \frac{|D_{\Gamma_t}v|^2}{v}\,d\mathscr{H}^n
   +\int_{\Gamma_t}\bar{\rho} v^3(\vec{f}\cdot\vec{n})^2\,d\mathscr{H}^n. \\
  \end{split}
 \end{equation}

 In the following, we derive the integral estimates for the transport terms.
 Using the H\"older inequality,
 \[
 \begin{split}
  &\quad\left|\int_0^\tau\eta\,dt\int_{\Gamma_t}
  \bar{\rho}(\vec{f}\cdot\vec{n})^2v^3
  \,d\mathscr{H}^n \right| \\
  &\leq\|\eta\|_{L^\infty(0,T_0)}
  \|v\|_{L^\infty(\Omega\times(0,\tau))}^3
  \left(\int_0^s\,dt
  \left(\int_{\Gamma_t}\bar{\rho}^{p'}\,d\mathscr{H}^n\right)^\frac{q'}{p'}
  \right)^\frac1{q'}
  \|\vec{f}\|^2_{L^q_tL^p_x(0,\tau,\Gamma_t)},
 \end{split}
 \]
 where $\frac2p+\frac1{p'}=1$ and $\frac2q+\frac1{q'}=1$.
 Using the convexity of $\Omega$,
 $|\tilde{X}-Y|\geq|X-Y|$; hence, 
 \[
 \begin{split}
  \int_{\Gamma_t}\bar{\rho}^{p'}\,d\mathscr{H}^n
  &\leq \frac{2^{p'-1}}{(4\pi(s-t))^{\frac{np'}{2}}}
  \int_{\Gamma_t}\exp\left(-\frac{p'|X-Y|^2}{4(s-t)}\right)\,d\mathscr{H}^n \\
  &\leq \Cr{c-4.2}(s-t)^{-\frac{np'}{2}+\frac{n}{2}},
 \end{split} 
 \]
 where $\Cl[const]{c-4.2}>0$ is some constant depending only on $n$ and
 $p$. Therefore,
 \[
 \left(\int_0^s\,dt
 \left(\int_{\Gamma_t}\bar{\rho}^{p'}\,d\mathscr{H}^n\right)^\frac{q'}{p'}
 \right)^\frac1{q'} <\infty
 \]
 if $-\frac{nq'}2+\frac{nq'}{2p'}>-1$, which provides
 $1-\frac{n}{p}-\frac{2}{q}>0$. Using \eqref{eq:4.2} and \eqref{eq:4.3}, we
 obtain
\begin{equation}
 \label{eq:4.6}
 \begin{split}
  \int_0^\tau\eta(|I_1|+|I_2|)\,dt
  &\leq
  \frac12
  \int_0^\tau\eta\,dt
  \int_{\Gamma_t}
  \bar{\rho}|A_t|^2v
  \,d\mathscr{H}^n \\
  &\quad+\int_0^\tau\eta\,dt\int_{\Gamma_t}
  \bar{\rho}\frac{|D_{\Gamma_t}v|^2}{v}
  \,d\mathscr{H}^n \\
  &\quad+
  \Cr{c-4.3}\|\eta\|_{L^\infty(0,T_0)}
  \|v\|_{L^\infty(\Omega\times(0,\tau))}^3
  \|\vec{f}\|^2_{L^q_tL^p_x(0,\tau,\Gamma_t)}
 \end{split}
\end{equation}
 for a positive constant $\Cl[const]{c-4.3}>0$ depending only on $n$,
 $p$, $q$ and $T_0$.

 Because
 \[
 |du\cdot d(\bar{\rho}(x,u,t))|
 =|du\cdot d\bar{\rho}+|du|^2\bar{\rho}_{x_{n+1}}|
 \leq v^2|D\bar{\rho}|,
 \]
 we obtain 
 \[
 |I_3|\leq\int_{\Gamma_t}v^2|D\bar{\rho}||\vec{f}\cdot\vec{n}|\,d\mathscr{H}^n.
 \]
 Then using the H\"older
 inequality,
\begin{equation}
 \label{eq:4.5}
 \begin{split}
  \int_0^\tau\eta|I_3|\,dt 
  &\leq
  \|\eta\|_{L^\infty(0,T_0)}
  \|v\|_{L^\infty(\Omega\times(0,\tau))}^2\\
  &\quad\times
  \left(\int_0^s\,dt
  \left(\int_{\Gamma_t}|D\bar{\rho}|^{p'}\,d\mathscr{H}^n\right)^\frac{q'}{p'}
  \right)^\frac1{q'}
  \|\vec{f}\|_{L^q_tL^p_x(0,\tau,\Gamma_t)},
 \end{split}
\end{equation}
 where $\frac1p+\frac1{p'}=1$ and $\frac1q+\frac1{q'}=1$.
 Using the convexity of $\Omega$,
 \[
 |D\bar{\rho}|\leq \Cr{c-4.4}\frac{1}{(s-t)^{\frac12+\frac{n}2}}\exp\left(-\frac{|X-Y|^2}{8(s-t)}\right),
 \]
 where $\Cl[const]{c-4.4}>0$ is some constant depending only on $n$ and
 $p$. Therefore,
 \[
 \int_{\Gamma_t}|D\bar{\rho}|^{p'}\,d\mathscr{H}^n
 \leq \Cr{c-4.4}(s-t)^{-\frac{p'}2-\frac{np'}{2}+\frac{n}{2}};
 \]
 hence,
 \[
 \left(\int_0^s\,dt
  \left(\int_{\Gamma_t}|D\bar{\rho}|^{p'}\,d\mathscr{H}^n\right)^\frac{q'}{p'}
 \right)^\frac1{q'}
 <\infty
 \]
 if $-\frac{q'}2-\frac{nq'}{2}+\frac{nq'}{2p'}>-1$, which provides
 $1-\frac{n}p-\frac2q>0$. Therefore, using \eqref{eq:4.5} we obtain
 \begin{equation}
  \label{eq:4.7}
   \int_0^\tau\eta|I_3|\,dt
   \leq \Cr{c-4.6}
   \|\eta\|_{L^\infty(0,T_0)}
  \|v\|_{L^\infty(\Omega\times(0,\tau))}^2
  \|\vec{f}\|_{L^q_tL^p_x(0,\tau,\Gamma_t)}
 \end{equation}
 for some constant $\Cl[const]{c-4.6}>0$ depending only on $n$,
 $p$, $q$ and $T_0$.
 Combining \eqref{eq:4.6} and \eqref{eq:4.7}, we obtain \eqref{eq:4.1}.
\end{proof}

\begin{proof}%
 [Proof of Theorem \ref{thm:1}]
 Let $T\in(0,T_0)$. 
 We denote
 \[
 M_T := \sup _{0<t <T} \| v(\cdot,t) \|
 _{L^\infty(\Omega)}.
 \]
 
 We first consider the interior gradient estimates. By arguments similar
 to that in Proposition \ref{prop:3.6} and Lemma \ref{lem:4.1} with
 $\eta\equiv1$, for $Y=(y,y_{n+1})\in (\Omega\setminus N_{R/6}) \times\R$ and
 $0<\tau <s\leq T$,
 \begin{equation}
  \label{eq:4.8}
  \begin{split}
   &\quad\frac{d}{dt}\int_{\Gamma_t}v\rho_1 \,d\mathscr{H}^n \\
   &\leq -\int_{\Gamma_t} \rho_1
   \left(|A_t|^2v
   +\frac{2|D_{\Gamma_t}v|^2}{v}
   -du\cdot d(\vec{f}\cdot\vec{n})\right)
   \,d\mathscr{H}^n \\
   &\quad
   +\frac14\int_{\Gamma_t}v \rho_1 (\vec{f}\cdot\vec{n})^2
   \,d\mathscr{H}^n
   +\Cr{c-3.1}\mathscr{H}^n(\Gamma _t)
 \end{split}
 \end{equation}
 and
 \begin{equation}
  \label{eq:4.9}
   \begin{split}
    &\ \int_0^\tau\,dt\int_{\Gamma_t}\rho_1
    du\cdot d(\vec{f}\cdot\vec{n})
    \,d\mathscr{H}^n
    +\frac14\int_0^\tau\,dt\int_{\Gamma_t}v\rho_1(\vec{f}\cdot\vec{n})^2
    \,d\mathscr{H}^n \\
    &\leq \frac12
    \int_0^\tau\,dt
    \int_{\Gamma_t}
    \rho_1|A_t|^2v
    \,d\mathscr{H}^n 
    +\int_0^\tau\,dt\int_{\Gamma_t}
    \rho_1\frac{|D_{\Gamma_t}v|^2}{v}
    \,d\mathscr{H}^n \\
    &\quad+\Cr{c-4.1}\|v\|_{L^\infty(\Omega\times(0,\tau))}^3
    \|\vec{f}\|_{L^q_tL^p_x(0,\tau,\Gamma_t)}(1+\|\vec{f}\|_{L^q_tL^p_x(0,\tau,\Gamma_t)})
   \end{split} 
 \end{equation}
 where the positive constants $\Cr{c-3.1}$ and $\Cr{c-4.1}$ are same as
 in Lemma \ref{lem:3.4} and Lemma \ref{lem:4.1} respectively.
Using \eqref{eq:4.8} and \eqref{eq:4.9} we have
 \begin{equation*}
  \begin{split}
   &\quad\int_{\Gamma_\tau}v(x,\tau)\rho_1(X,\tau) \,d\mathscr{H}^n 
   -\int_{\Gamma_0}v(x,0)\rho_1(X,0) \,d\mathscr{H}^n  \\
   &\leq \Cr{c-3.1} \int _ 0 ^{\tau} \int _\Omega v \, dx dt
   +\Cr{c-4.1}\|v\|_{L^\infty(\Omega\times(0,\tau))}^3\|\vec{f}\|_{L^q_tL^p_x(0,\tau,\Gamma_t)}(1+\|\vec{f}\|_{L^q_tL^p_x(0,\tau,\Gamma_t)})
   \\
   &\leq \Cr{c-3.1}|\Omega| M_Ts
   +\Cr{c-4.1}M_T^3\|\vec{f}\|_{L^q_tL^p_x(0,s,\Gamma_t)}(1+\|\vec{f}\|_{L^q_tL^p_x(0,s,\Gamma_t)}).
  \end{split}
 \end{equation*}
 Passing $\tau\to s$ we obtain
 \begin{equation}
  \label{eq:4.11}
\begin{split}
   v(y,s)
   &\leq \int _\Omega v^2 _0 \rho_1 (X,0) \, dx
   +\Cr{c-3.1}|\Omega| M_Ts
   \\
   &\quad+\Cr{c-4.1}M_T^3\|\vec{f}\|_{L^q_tL^p_x(0,s,\Gamma_t)}(1+\|\vec{f}\|_{L^q_tL^p_x(0,s,\Gamma_t)})
   \\
   &\leq \|v_0\|_\infty^2
   +\Cr{c-3.1}|\Omega| M_Ts
   +\Cr{c-4.1}M_T^3\|\vec{f}\|_{L^q_tL^p_x(0,s,\Gamma_t)}
   (1+\|\vec{f}\|_{L^q_tL^p_x(0,s,\Gamma_t)})
  \end{split} 
 \end{equation}
 where $v_0:=\sqrt{1+|du_0|^2}$.
  
 Next we consider the boundary gradient estimates. By Proposition
 \ref{prop:3.6}, for $Y=(y,y_{n+1})\in N_{R/4}\times\R$, $0<\tau<s\leq T$, and
 a non-negative function $\eta=\eta(t)\in L^\infty(0,T_0)$,
 \begin{equation*}
  \begin{split}
   &\quad\eta\frac{d}{dt}\int_{\Gamma_t}v(\rho_1+\rho_2)\,d\mathscr{H}^n \\
   &\leq -\eta\int_{\Gamma_t}(\rho_1+\rho_2)
   \left(|A_t|^2v
   +\frac{2|D_{\Gamma_t}v|^2}{v}
   -du\cdot d(\vec{f}\cdot\vec{n})\right)
   \,d\mathscr{H}^n \\
   &\quad
   +\frac\eta4\int_{\Gamma_t}v(\rho_1+\rho_2)(\vec{f}\cdot\vec{n})^2
   \,d\mathscr{H}^n \\
   &\quad
   +\Cr{c-3.1}\eta\mathscr{H}^n(\Gamma _t)
   +\Cr{c-3.2}\eta(s-t)^{-\frac34}\int_{\Gamma_t}v(\rho_1+\rho_2)
   \,d\mathscr{H}^n \\
   &\quad
   +\Cr{c-3.3}\eta\int_{\Gamma_t\cap\supp\rho_2}v
   \,d\mathscr{H}^n. 
  \end{split}
 \end{equation*}
 Let
 \[
 \eta(\tau)
 :=\exp\left(-\Cr{c-3.2}\int_0^\tau(s-t)^{-\frac34}\,dt\right)
 =\exp\left(-\Cr{c-3.2}(s^\frac14-(s-\tau)^\frac14)\right).
 \]
 Note that $\|\eta\|_\infty=1$. Then by Lemma \ref{lem:4.1}, we have
 \begin{equation*}
  \begin{split}
   &\quad
   \exp\left(-\Cr{c-3.2}(s^\frac14-(s-\tau)^\frac14)\right)
   \int_{\Gamma_\tau}v(x,\tau)(\rho_1+\rho_2)(X,\tau) \,d\mathscr{H}^n \\
   &\qquad-\int_{\Gamma_0}v(x,0)(\rho_1+\rho_2)(X,0) \,d\mathscr{H}^n  \\
   &\leq \Cr{c-3.1} \int_ 0^{\tau}\eta\,dt
   \int _\Omega v \, dx
   +\Cr{c-3.3}\int_0^\tau\eta\,dt\int_{\Omega}v^2\,dx \\
   &\quad+\Cr{c-4.1}\|v\|_{L^\infty(\Omega\times(0,\tau))}^3
   \|\vec{f}\|_{L^q_tL^p_x(0,\tau,\Gamma_t)}
   (1+\|\vec{f}\|_{L^q_tL^p_x(0,\tau,\Gamma_t)})
   \\
   &\leq \Cr{c-3.1}|\Omega| M_Ts
   +\Cr{c-3.3}|\Omega|M_T^2s
   +\Cr{c-4.1}M_T^3\|\vec{f}\|_{L^q_tL^p_x(0,s,\Gamma_t)}(1+\|\vec{f}\|_{L^q_tL^p_x(0,s,\Gamma_t)}).
  \end{split}
 \end{equation*}
 Passing $\tau\to s$ we obtain
 \begin{equation}
  \label{eq:4.10}
   \begin{split}
    \exp\left(-\Cr{c-3.2}s^\frac14\right)
    v(y,s)
   &\leq 2\|v_0\|_\infty^2
    +\Cr{c-3.1}|\Omega| M_Ts
    +\Cr{c-3.3}|\Omega|M_T^2s
   \\
    &\quad+
    \Cr{c-4.1}M_T^3\|\vec{f}\|_{L^q_tL^p_x(0,s,\Gamma_t)}(1+\|\vec{f}\|_{L^q_tL^p_x(0,s,\Gamma_t)})
   \end{split} 
 \end{equation}
 where $v_0:=\sqrt{1+|du_0|^2}$.
 Compared \eqref{eq:4.11} with \eqref{eq:4.10}, we obtain for all
 $y\in\overline{\Omega}$ and $0<s\leq T$
  \begin{equation}
   \label{eq:4.12}
    \begin{split}
     \exp\left(-\Cr{c-3.2}s^\frac14\right)
     v(y,s)
     &\leq 2\|v_0\|_\infty^2
     +\Cr{c-4.5}(s)M_T^3
    \end{split} 
  \end{equation}
 where
 \begin{equation*}
  \Cl[const]{c-4.5}(s)
   :=(\Cr{c-3.1}+\Cr{c-3.3})|\Omega|s
   +\Cr{c-4.1}\|\vec{f}\|_{L^q_tL^p_x(0,s,\Gamma_t)}
   (1+\|\vec{f}\|_{L^q_tL^p_x(0,s,\Gamma_t)}).
 \end{equation*}
 Note that $\Cr{c-4.5}(s)$ is monotone increasing and
 $\Cr{c-4.5}(s)\rightarrow0$ as $s\downarrow 0$.
 
 Now, select $(y,s)$ such that $M_T=v(y,s)$ and $Y=(y,u(y,s))$. Then,
 by monotonicity of $\Cr{c-4.5}(s)$
 \begin{equation}
  \label{eq:4.13}
   \Cr{c-4.5}(T)M_T^3
   -\exp\left(-\Cr{c-3.2}T^\frac14\right)M_T
   +2\|v_0\|_\infty^2\geq0.
 \end{equation}
 When $4\|v_0\|_\infty^2\leq M_T\leq 5\|v_0\|_\infty^2$ for some
 $T>0$. Then by \eqref{eq:4.13} we have
 \[
  \Cr{c-4.5}(T)
  \geq \frac{\exp\left(-\Cr{c-3.2}T^\frac14\right)}{M_T^2}
  -\frac{2\|v_0\|_\infty^2}{M_T^3} \\
 \geq \frac{\Cr{c-4.7}(T)}{\|v_0\|_\infty^4}
\]
 where
 \[
 \Cl[const]{c-4.7}(T)=\frac{\exp\left(-\Cr{c-3.2}T^\frac14\right)}{25}
 -\frac{1}{32}.
 \]
 Thus, we have
 \[
  M_{T_1}\leq 4\|v_0\|_\infty^2=1+\|du_0\|_\infty^2
 \]
 where $T_1$ is sufficiently small constant satisfying
 $\Cr{c-4.7}(T)>0$ and $ \Cr{c-4.5}(T)<\frac{\Cr{c-4.7}(T)}{\|v_0\|_\infty^4}$.
\end{proof}

\section{Existence of classical solutions}
\label{sec:5}
Finally, we prove Theorem $\ref{thm:2}$. To use the Schauder estimates, we
provide the following:
\begin{lemma}
 Let $T>0$ and $u\in C^{2,1}(Q_T )$ be a solution of $(\ref{eq:1.1})$. Then
 \begin{equation}
  \sup _{Q_T}|u| \leq \sup_{\Omega \times \mathbb{R} \times[0,1]}|\vec{f}| \, T + \sup _{\Omega}| u_0 |.
   \label{supu}
 \end{equation}
\end{lemma}
\begin{proof}
We set $w(x,t)=\sup_{\Omega \times \mathbb{R} \times[0,1]}|\vec{f}| \, t + \sup _{\Omega}| u_0 |$. We note that
\[ \partial_t w \geq \sqrt{1+|dw|^2} \Div \left( \frac{dw}{\sqrt{1+|dw|^2}} \right) + \vec{f}(x,w,t)\cdot (-dw,1).  \]
Using the comparison principle, we determine that
\[w\geq u, \qquad (x,t) \in Q_T. \]
Similarly to the above argument, 
\[u\geq -w, \qquad (x,t) \in Q_T. \]
Hence, we obtain $(\ref{supu})$.
\end{proof}

\begin{proof}%
 [Proof of Theorem $\ref{thm:2}$]
 Fix $\alpha \in (0,1)$. We assume that $u_0 \in C^{2,\alpha}(\Omega)$
 and let $T>0$, which is given by Theorem \ref{thm:1}. 
 Let $\beta \in (0,\alpha]$ and we set $X:= C^{1,\beta }(Q_T )$.
 We consider the following linear parabolic type equation:
 \begin{equation}
  \left\{
   \begin{aligned}
    \partial _t u &= \sum _{i,j=1} ^{n} a_{ij}(dw) \partial _{x_i x_j} u + \vec{f}(x,w,t) \cdot (-du,1)& , \qquad &\text{in} \ Q_T  , \\
    du\cdot \nu \Big|_{\partial \Omega} &=0,\\
    u\Big|_{t=0}&=u_0 ,& \qquad &\text{on} \ \Omega,
   \end{aligned}
 \right.
 \label{mcf3}
 \end{equation}
 where $w \in X$ and $\displaystyle a_{ij}(r)=\left( \delta _{ij}
 -\frac{r_i r_j }{1+|r|^2} \right)$ for $r=(r_1,\dots , r_{n})$. Because 
 
 \begin{equation}
  \begin{split}
   \|a_{ij}(dw)\|_{C^{\alpha\beta}(Q_T)}
   &\leq\|a_{ij}(dw)\|_{C^{\beta}(Q_T)} \\
   &\leq\|a_{ij}\|_{C^1 (\mathbb{R}^n)} \|dw\|_{C^\beta (Q_T)}
   \leq \| a_{ij}\|_{C^1 (\mathbb{R}^n)} \|w\|_{X}
  \end{split}
  \label{aij}
 \end{equation}
 for any $w\in X$, \eqref{mcf3} is uniformly parabolic in $Q_T$. Note
 that $\|a_{ij}\|_{C^1 (\mathbb{R}^n)} <\infty$. 
 Using 
 \eqref{eq:2.1}, we obtain
 \begin{equation}
 \|\vec{f}(\cdot,w,\cdot) \|_{C^{\alpha\beta}(Q_T)}
  \leq K\|w\|_{C^\beta (Q_T)}
  \leq K\|w\|_{X}
 \label{fv}
 \end{equation}
 for any $w \in X$. Hence, for any $w\in X$ there exists a unique solution $u_w \in C^{2,\alpha\beta}(Q_T) \subset X$ of $(\ref{mcf3})$ such that
 \begin{equation}
 \|u_w\| _{C^{2,\alpha\beta }(Q_T)}\leq \Cr{c-5.1},
 \label{schauder2}
 \end{equation}
 where $\Cl[const]{c-5.1}>0$ depends only on 
 $n,\alpha, \beta, \|w\|_{X}, \|u_0\|_{C^{2,\alpha}(\Omega)}$ 
 and $K$ (see \cite[Theorem 4.5.3]{MR0241822}).

 We define $A :X\to X$ as $A w=u_w$. Note that $A$ is continuous and compact. We show that
 \[ S :=\{ u \ | \ u = \sigma A u \ \text{in} \ X, \ \text{for some} \ \sigma \in [0,1] \} \]
 is bounded in $X$. If $u\in S $, then
 \begin{equation}
 \left\{
 \begin{aligned}
  \partial _t u &= \sum _{i,j=1} ^{n} a_{ij}(du) \partial _{x_i x_j} u + \vec{f}(x,u,t) \cdot (-du,\sigma ),& \qquad &\text{in} \ Q_T  , \\
 du\cdot \nu \Big|_{\partial \Omega} &=0,&&\\
 u\Big|_{t=0}&= \sigma u_0,& \qquad &\text{on} \ \Omega.
 \end{aligned}
 \right.
 \label{mcf4}
 \end{equation}
 According to Theorem \ref{thm:1}, 
 \begin{equation}
  \sup_{Q_T} |du| \leq 4(1+\| du_0\|_\infty^2).
   \label{estc1}
 \end{equation}
 Because $du\cdot\nu=0$ on $\partial \Omega$, we can use similar
 arguments to the interior Schauder estimates (cf. \cite[Theorem
 6.2.1]{MR0241822}); hence,
 \begin{equation}
 \|du\|_{C^\beta (Q_T) } \leq \Cr{c-5.3},
 \label{estc2}
 \end{equation}
 where $\Cl[const]{c-5.3}=\Cr{c-5.3}$ is a positive constant depending
 only on $n$, $\sup _{Q_T}|u|$, $\|du_0 \|_{C^\alpha(\Omega)}$,
 $\sup_{\Omega \times \mathbb{R} \times [0,T]} |\vec{f}|$, and $\partial
 \Omega$.  

 Using the same argument as \eqref{schauder2},
 \begin{equation}
 \|u\|_{X}\leq \|u\|_{C^{2,\alpha\beta}(Q_T)}\leq \Cr{c-5.4},
 \label{estc3}
 \end{equation}
 where $\Cl[const]{c-5.4}=\Cr{c-5.4}  (n,\alpha, \|u_0\|_{C^{2,\alpha} (\Omega)}, \Cr{c-5.3} , K)>0$ (see ~\cite{MR0241822}). According to \eqref{estc1}, \eqref{estc2}, and \eqref{estc3}, $\Cr{c-5.4}$ depends only on $n,\alpha,\|u_0 \|_{C^{2,\alpha} (\Omega)}, \sup _{\Omega} |du_0|$ and $K$. Thus, $S$ is bounded in $X$.
 According to the Leray-Schauder fixed point theorem, there exists a solution $u \in C^{ 2,\alpha}(Q_T )$ of $(\ref{eq:1.1})$. 
 
 We return to the assumption that $u_0$ is a Lipschitz function with
 a Lipschitz constant $L>0$. Set $\varepsilon >0$. We choose smooth
 functions $u_0 ^k$ converging uniformly to $u_0$ on $\Omega$. We note
 that according to Theorem $\ref{thm:1}$, 
 
\[
 \sup_{Q_T}|du ^k|\leq 4(1+L^2)
\]
 for all $k\geq 1$. Using an argument similar to
 \eqref{estc2}, \eqref{estc3} and the interior Schauder estimates, there
 exists $\Cl[const]{c-5.5}=\Cr{c-5.5}(n,\alpha,L,\varepsilon,K)>0$ such that
 \begin{equation*}
 \sup _k 
  \| u^k \|_{C^{ 2,\alpha}(Q_T ^\varepsilon)}
  \leq  \Cr{c-5.5}.
 \end{equation*}
 where $u^k$ is the solution of $(\ref{eq:1.1})$ with $u^k (x,0)=u_0
 ^k(x)$ in $\Omega$. Note that $\varepsilon=\dist (Q^\varepsilon
 _T,\partial Q_T)$.

 Hence, for any $\varepsilon >0$, passing to a subsequence if necessary,
 $\{ u^k \}_{k=1} ^{\infty}$ converges to a classical solution $u$ in
 $Q_T ^\varepsilon$ and we obtain \eqref{schauder}. Therefore, by
 diagonal arguments, we obtain the solution $u \in C(\overline{Q_T})\cap
 C^{2,1} (Q_T ^\varepsilon)$ of \eqref{eq:1.1}. The comparison principle
 implies the uniqueness of the solution of \eqref{eq:1.1}. Thus, we have
 proved Theorem $\ref{thm:2}$.
\end{proof}

\section{Optimality}
\label{sec:6}

We now discuss optimality about the assumption to transport terms for
the gradient estimates. We present some transport term $\vec{f}$ such
that the gradient of some solution
of \eqref{eq:1.1} is not bounded and
$\|\vec{f}\|_{L^q_tL^p_x(0,1,\Gamma_t)}<\infty$ for some $p$, $q$
satisfying $\frac{n}{p}+\frac{2}{q}>1$(see Figure \ref{fig:6.1}).  

\begin{figure}
 \centering \includegraphics[height=120pt]{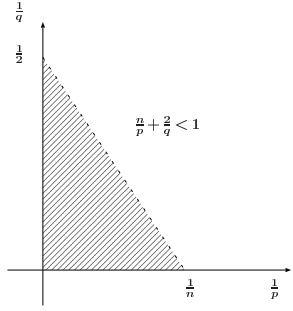} \caption{For
 $p$ and $q$ \emph{not} belonging to the above oblique line, there is a
 transport term such that the gradient of some solution of
 \eqref{eq:1.1} is not bounded.}  \label{fig:6.1}
\end{figure}

Let $\phi=\phi(\xi)$ be a smooth function on $\R^n$ compactly supported
on $\Omega$. We seek the transport term such that
$u(x,t)=(1-t)^\alpha\phi(x/\sqrt{1-t})$ is a solution of \eqref{eq:1.1}
for $\alpha\in\R$. By direct calculation, we obtain
\[
 \begin{split}
  du(x,t)&=(1-t)^{\alpha-\frac12}d_\xi\phi,\quad
  d^2u(x,t)=(1-t)^{\alpha-1}d^2_\xi\phi \\  
  \partial_tu(x,t)&=(1-t)^{\alpha-1}\left(-\alpha\phi
  +\frac12 d_\xi\phi\cdot\frac{x}{\sqrt{1-t}}
  \right). 
 \end{split}
\]
Note that if $\alpha<\frac12$, then $du$ is not bounded.  If
$u(x,t)=(1-t)^\alpha\phi(x/\sqrt{1-t})$ is a solution of \eqref{eq:1.1},
then we obtain
\begin{equation}
 \label{eq:6.1}
  \begin{split}
  \vec{f}(x,u,t)\cdot(-du,1)
  &=(1-t)^{\alpha-1}\left(-\alpha\phi
  +\frac12 d_\xi\phi\cdot\frac{x}{\sqrt{1-t}}
  \right) \\
  &\quad-(1-t)^{\alpha-1}\Delta_\xi\phi
  +(1-t)^{3\alpha-2}
  \frac{d_\xi^2\phi:(d_\xi\phi\otimes
   d_\xi\phi)}{1+(1-t)^{2\alpha-1}|d_\xi\phi|^2}. 
 \end{split}
\end{equation}
Thus, for $p,q\geq1$ and $\alpha<\frac12$ if $t$ is sufficiently
near to $1$, we obtain
\begin{equation*}
 \Cr{c-6.1}(\phi)(1-t)^{3\alpha-2+\frac{n}{2p}+\frac{2\alpha-1}{p}}
 \leq \left(\int_{\Gamma_t}|\vec{f}|^p\,d\mathscr{H}^n\right)^\frac1p 
 \leq\Cr{c-6.2}(\phi)(1-t)^{3\alpha-2+\frac{n}{2p}+\frac{2\alpha-1}{p}}
\end{equation*}
for some constants $\Cl[const]{c-6.1}(\phi), \Cl[const]{c-6.2}(\phi)>0$
since $\frac1{2(1-t)^{2\alpha-1}|d\phi|^2}\leq
\frac1{1+(1-t)^{2\alpha-1}|d\phi|^2}\leq \frac1{(1-t)^{2\alpha-1}|d\phi|^2}$.

When $1+\varepsilon_0:=\frac{n}{p}+\frac{2}{q}>1$ and
$\|\vec{f}\|_{L^q_tL^p_x(0,1,\Gamma_t)}<\infty$, then 
\[
 \int_0^1(1-t)^{(3\alpha-2+\frac{n}{2p}+\frac{2\alpha-1}{p})q}\,dt<\infty
\]
hence we obtain $\alpha>\frac12-\frac{\varepsilon_0}{2}(3+\frac2p)^{-1}$. For
$\alpha_0:=\frac12-\frac{\varepsilon_0}{4}(3+\frac2p)^{-1}$, let $\vec{f}$ be given by
\eqref{eq:6.1} with $\alpha=\alpha_0$. Then
$u(t,x):=(1-t)^\alpha\phi(x/\sqrt{1-t})$ is a solution of
\eqref{eq:6.1}, $\|\vec{f}\|_{L^q_tL^p_x(0,1,\Gamma_t)}<\infty$ for some
$p$, $q$ satisfying $\frac{n}{p}+\frac{2}{q}>1$ and $du$ is not bounded
on $\Omega\times(0,1)$.

\section{Final remark}

By our argument, we do not obtain a time global classical solution of
\eqref{eq:1.1}. Indeed, when $f\not\equiv0$, the maximum existence time
in Theorem \ref{thm:2} cannot be taken infinity.  On the other hand,
Huisken~\cite[Theorem 1.1]{MR0983300} proved that there exists a time
global solution of \eqref{eq:1.1} under $\vec{f}\equiv 0$ and $u_0 \in
C^{2,\alpha}(\overline{\Omega}) $ and the solution converges to some
constant function as $t\to \infty$. In the case of $\vec{f} \not \equiv
0$, a priori time global gradient bounds is not known hence we do not
show the global existence for solutions of \eqref{eq:1.1}. It is
expected that with the assumption \eqref{eq:2.1} there is a time global
solution of \eqref{eq:1.1} and that solution converges to a solution of
the prescribed mean curvature equation $\vec{H}=-(\vec{f} \cdot \vec{n})
\vec{n}$.

To apply our results for $\Gamma_t$ in \eqref{eq:1.4}, the velocity
vector $\vec{f}$ in \eqref{eq:1.4} needs to be smooth enough. On the
other hand, if $\vec{f}$ is a weak solution of \eqref{eq:1.4}, namely
$\vec{f} \in L^2_t(H^1_x)$, then $\vec{f}$ does not satisfy our
assumption. When we consider (1.4) with $\vec{f} \in L^2_t(H^1_x)$, we
need to study the relationship between the velocity vector $\vec{f}$ and
the phase boundary $\Gamma_t$.

\section*{Acknowledgments}
The authors are grateful to the referee for his or her helpful comments.
This work was supported by JSPS KAKENHI Grant Numbers 25800084,
25247008, 16K17622, 17J02386. The second author is supported by JSPS
Research Fellowships for Young Scientists.

\end{document}